\documentclass{amsart}
\usepackage{verbatim}
\usepackage{amssymb}
\usepackage{amsbsy}
\usepackage{amscd}
\usepackage{amsmath}
\usepackage{amsthm}
\usepackage[mathscr]{eucal}

\newtheorem*{theorem*}{Theorem}
\newtheorem{proposition}{Proposition}

\def\R{\mathbb{R}}
\def\C{\mathbb{C}}
\def\D{\mathbb{D}}
\def\Cont{\mathcal{C}}
\def\CP{\mathbb{CP}}
\def\Aut{\mathrm{Aut}}

\def\Ker{\mathrm{Ker}}
\def\id{\mathrm{Id}}

\newcommand{\eqdef}{\mathrel{\mathop:}=}

\def\d'{\partial}

\def\bd{\partial}
\newcommand{\pa}{\partial}
\newcommand{\opa}{\overline\pa}
\newcommand{\ol}{\overline }

\renewcommand\Re{\operatorname{Re}}

\title[The Diederich--Fornaess index of Levi-flats]{On a global estimate of the Diederich--Fornaess index of Levi-flat real hypersurfaces}
\author{Masanori Adachi}
\address{%
Graduate School of Mathematics, Nagoya University, 
Furo-cho Chikusa-ku Nagoya 464-8602, Japan
}
\email{m08002z@math.nagoya-u.ac.jp}
\curraddr{Center~for~Geometry~and~its~Applications, Pohang~University~of~Science~and~Technology, Pohang 790-784, Republic of Korea
}

\thanks{The author is partially supported by JSPS Grant-in-Aid for Young Scientists (B) 26800057.}

\begin{document}

\begin{abstract}
In this expository paper, we review a recent progress of the study of the Diederich--Fornaess index of complex domains
with emphasis on the case of domains with Levi-flat boundary. 
It is exhibited that for any compact Levi-flat real hypersurface, 
the norm of its infinitesimal holonomy must exceed the curvature of its normal bundle at a point.
\end{abstract}

\maketitle

\section{A perspective on Levi-flats}
The study of Levi-flat real hypersurfaces has attracted both foliators and complex analysts since 1980s. 
We may say that its main driving force is the exceptional minimal set conjecture, in particular, 
the non-existence conjecture of smooth Levi-flat real hypersurface in the complex projective plane; 
this conjecture is still open despite many attempts. 

Besides this particular problem, the Levi-flat real hypersurface itself has been recognized as a quite attractive object 
within which we expect to find very subtle interaction between the theory of foliation and several complex variables. 
Around 1990, the works of Barrett and Inaba revealed this situation (see \cite{barrett1990}, \cite{barrett1992}, \cite{barrett-inaba1992}, \cite{inaba1992}). 
For example, the striking achievement in \cite{barrett1990} is that 
\textit{the standard $\Cont^\infty$-smooth Reeb foliation of $S^3$ cannot be realized as a $\Cont^\infty$-smooth Levi-flat real hypersurface in any complex surface}. 
This result not only clarifies the difference of Levi-flat manifolds and Levi-flat real hypersurfaces, 
but the method of its proof exhibits a beautiful interaction between two research fields: 
once we had a realization, 
the theory of Ueda \cite{ueda1982} allows us to connect the holonomy around its compact leaf and the pseudoconvexity of its neighborhood, 
from which we can deduce a contradiction. 
The aim of this expository paper is to illustrate a similar interaction in a different situation.

Now let us start with recalling what a Levi-flat manifold is. 
A $\Cont^k$-smooth manifold $M$ of dimension $(2n + 1)$ is said to be \textit{Levi-flat}
if it has a non-singular $\Cont^k$-smooth foliation $\mathcal{F}$ by $n$-dimensional complex leaves. 
More precisely, it is equipped with a $\Cont^k$-smooth foliated atlas $\{(U_\lambda, (z_\lambda, t_\lambda)) \}$ of $M$ such that 
\begin{itemize}
\item $(z_\lambda, t_\lambda) : U_\lambda \to \D^n \times I$ where $\D^n$ is the leaf direction and $I$ is the transverse direction.
 Here we denoted $\D = \{ z \in \C \mid |z| < 1 \}$ and $I = (-1,1)$.
\item On any intersecting charts $U_\lambda$ and $U_{\lambda'}$, the coordinate changes $z_{\lambda'} = z_{\lambda'}(z_\lambda, t_\lambda)$, 
$t_{\lambda'} = t_{\lambda'}(t_\lambda)$ are not only $\Cont^k$-smooth but also holomorphic in $z_\lambda$. 
\end{itemize}
We call $\mathcal{F}$ the {\em Levi foliation} of $M$. 
Note that any 3-manifold with an oriented foliation of codimension one has a Levi-flat structure; 
We can give it a leafwise Riemannian metric and induce isothermal coordinates on the leaves. 

A Levi-flat real hypersurface is a Levi-flat manifold realized in a complex manifold as a real hypersurface. 
More precisely, for $k \geq 2$, a $\Cont^k$-smooth closed real hypersurface $M$ in a complex manifold $X$ is said to be {\em Levi-flat} 
if it is foliated by complex hypersurfaces of $X$. 
The integrable distribution $T\mathcal{F}$ of its Levi foliation $\mathcal{F}$ is given by the maximal complex subbundle $TM \cap J_X TM$ of $TM$ 
where $J_X$ denotes the complex structure of $X$. 
We often identify $T\mathcal{F} \subset TX|M$ with $T^{1,0}\mathcal{F} \subset T^{1,0}X|M$ in the standard way. 
When we locally express $M$ on a neighborhood $U$ of $p \in M$ as a zero set of a real-valued $\Cont^k$-smooth function $\rho$ 
using the implicit function theorem, 
the Levi-flatness of $U \cap M$ is equivalent to the vanishing of a quadratic form called the \textit{Levi form of $\rho$ along $U \cap M$}, 
namely, 
\[
i\pa\opa \rho (\zeta, \ol{\zeta}) = 0 \text{\quad for any $\zeta \in T^{1,0}_q \mathcal{F} = \Ker (\pa\rho)_q, q \in U \cap M$}.
\]
Note that Barrett and Forn{\ae}ss showed in \cite{barrett-fornaess1988} that 
the Levi foliation is actually of $\Cont^k$ while the Frobenius' theorem only proves it is of $\Cont^{k-1}$. 

If a possibly singular holomorphic foliation of codimension one on a complex manifold has a saturated smooth real hypersurface, it is Levi-flat. 
Conversely, if one has a real-analytic Levi-flat real hypersurface in a complex manifold, its Levi foliation 
extends to a non-singular holomorphic foliation defined on a neighborhood of the real hypersurface (see \cite{rea1972}).

\section{A quantitative estimate on Levi-flat real hypersurfaces}
Let us state the goal of this paper. It is to explain a theorem that follows from results 
obtained in \cite{brunella2008}, \cite{adachi_local}, \cite{ohsawa-sibony1998}, \cite{adachi-brinkschulte}, \cite{fu-shaw}.

\begin{theorem*}
Let $M$ be an oriented compact $\Cont^4$-smooth Levi-flat manifold of dimension $(2n + 1)$, $n \geq 1$. 
Denote its Levi foliation by $\mathcal{F}$.
Suppose $M$ can be realized in an $(n+1)$-dimensional complex manifold $X$ as the boundary of a relatively compact domain $\Omega \Subset X$.
Then, for any transverse measure $\mu$ of $\mathcal{F}$ with $\Cont^3$-smooth positive density with respect to transverse Lebesgue measures, 
there exists a non-zero vector $\zeta \in T^{1,0}\mathcal{F}$ at which $iA_\mu(\zeta, \ol{\zeta}) \geq n i\Theta_\mu(\zeta, \ol{\zeta})$ holds. 
\end{theorem*}

Our main ingredients, globally defined leafwise $(1,1)$-forms $iA_\mu$ and $i\Theta_\mu$ on $M$, are defined by 
their local expressions
\begin{align*}
iA_\mu (= iA_h) &=i\sum_{\alpha, \beta =1}^n \frac{\pa \log h}{\pa z^{\alpha}} \frac{\pa \log h}{\pa \ol{z}^\beta} dz^{\alpha} \wedge d\ol{z}^{\beta}, \\
i\Theta_\mu (= i\Theta_h) &=i\sum_{\alpha, \beta =1}^n \frac{\pa^2  (-\log h)}{\pa z^{\alpha} \pa \ol{z}^\beta} dz^{\alpha} \wedge d\ol{z}^{\beta}
\end{align*}
on each foliated chart $(z,t)$ where $h$ denotes the density of $\mu$ with respect to the Lebesgue measure $|dt|$: $\mu = h(z,t) |dt|$ on each transversal $\{z\} \times I$.
Note that we are identifying these leafwise $(1,1)$-forms with quadratic forms defined on $T^{1,0}\mathcal{F}$. 
The quadratic form $iA_\mu$ expresses a certain norm of the infinitesimal holonomy of the Levi foliation with respect to $\mu$.
On the other hand, $i\Theta_\mu$ can be regarded as the leafwise Chern curvature of the holomorphic normal bundle $N^{1,0}$ of $M$.
Here the \textit{holomorphic normal bundle} $N^{1,0}$ of $M$ is defined by
\[
N^{1,0} \eqdef (T^{1,0}X|M)/T^{1,0}\mathcal{F} \simeq \C \otimes TM/T\mathcal{F}
\]
and it is easily seen that $h^2$ defines a $\Cont^3$-smooth hermitian metric of $N^{1,0}$ thanks to its transition rule. 

We are going to give a sketch of a proof of Theorem in the following two sections.
The key idea is to translate $iA_\mu$ and $i\Theta_\mu$, geometric information of holonomy of $\mathcal{F}$, into 
the Diederich--Fornaess index, a numerical information of pseudoconvexity of the bounded domain $\Omega$
via the holomorphic normal bundle $N^{1,0}$ (see \S3). This idea is based on Brunella \cite{brunella2008}
and developed by the author in \cite{adachi_local}. 
After that, we will deduce the conclusion by relying on a global estimate of the Diederich--Fornaess index 
obtained independently by the author and Brinkschulte \cite{adachi-brinkschulte} and by Fu and Shaw \cite{fu-shaw} (see \S4).

\section{First half of the proof: Moving onto the complement}

The proof is by contradiction. Let us suppose the contrary:
for any non-zero $\zeta \in T^{1,0}\mathcal{F}$, $iA_\mu(\zeta, \ol{\zeta}) < n i\Theta_\mu(\zeta, \ol{\zeta})$ holds. 
This particularly implies that the holomorphic normal bundle $N^{1,0}$ is positive: $i\Theta_\mu$ defines a positive-definite
quadratic form on $T^{1,0}_p\mathcal{F}$ for every $p \in M$. We first apply a construction originating from Brunella \cite{brunella2008}. 

\begin{proposition}[cf. \cite{brunella2008}]
\label{brunella}
There exists a $\Cont^3$-smooth real-valued function $\delta$ defined on a neighborhood $U \supset M$
such that
\begin{enumerate}
\item $\delta$ is a \textit{$\Cont^3$-smooth defining function} of $M$, i.e., $0$ is a regular value of $\delta$ and $M = \delta^{-1}(0)$. 
\item There exists a hermitian metric $\omega$ of $X$ such that $i\pa\opa(-\log |\delta|) > \omega $ holds on $U \setminus M$.
\end{enumerate}
\end{proposition}

We give a simplified construction in our rather restricted setting.
\begin{proof}
First take a $\Cont^3$-smooth non-vanishing section of $TM/T\mathcal{F}$ and normalize it using $h^2$, the hermitian metric of $N^{1,0}$
induced from $\mu$, and lift it to a section of $TX|M$, say $\xi$. 
We rotate the transversal vector field $\xi$ by the complex structure $J_X$ and obtain a normal vector field $\nu = J_X \xi$ of $M$. 
Consider arbitrary $\Cont^3$-smooth extension of $\nu$ on a neighborhood of $M$, denoted by $\nu$ again. 

Now we integrate the vector field $\nu$ and obtain a $\Cont^3$-diffeomorphism, say $\phi: M \times (-\varepsilon, \varepsilon) \to \widetilde{U} \subset X$. 
Let $\delta = \mathrm{proj_2} \circ \phi^{-1}: \widetilde{U} \to (-\varepsilon, \varepsilon)$. 
It is clear that $\delta$ is a $\Cont^3$-smooth defining function of $M$.

The remaining problem is to check the condition (2). From what we supposed, it follows that $i\Theta_h$ defines a leafwise hermitian metric of $M$.
Using the smooth decomposition $T^{1,0}X|M \simeq T^{1,0}\mathcal{F} \oplus N^{1,0}$ defined by $\xi$, 
we can construct a hermitian metric of $T^{1,0}X$, say $\omega$, so that $\omega = (i\Theta_h \oplus i\pa \delta \wedge \opa\delta)/100$ on $M$. 
For this $\omega$, we will show $i\pa\opa(-\log |\delta|) > \omega$ on some smaller neighborhood $U$ of $M$. 

We will work locally to show this estimate, namely, show it on $\Gamma_p$ for each $p \in M$ 
where $\Gamma_p \subset \widetilde{U} \setminus M$ is a small neighborhood of the normal curve $\phi(p, t)$ $(0<|t| \ll 1)$ specified later. 
To accomplish it, we will exploit the \textit{distinguished parametrization} in \cite{adachi_local},
which allows us to treat $M$ locally as if it is a real-analytic Levi-flat real hypersurface when we compute leafwise objects. 
Now let us fix $p \in M$ and take a distinguished parametrization of $M \subset X$ in a local coordinate $(V, z)$ of $X$ around $p$, 
say $\varphi: \D^{n} \times I \hookrightarrow V$. 
(You may assume the real-analyticity of $M$ and just consider the natural inclusion $\varphi: \D^n \times I \hookrightarrow \D^{n+1}$ 
to understand the essence of this proof.)
We can assume that $\varphi_*((\pa_t)_{(0,0)}) = \xi_p$ by a linear transformation where we denote the transversal coordinate $t \in I$. 
We will choose a sufficiently small neighborhood $\Gamma_p$ of $\{ (0, iy_{n+1}) \in V \subset \C^n \times \C \mid 0 < |y_{n+1}| \ll 1 \}$ in $V \setminus M$ later.

We start to estimate $i\pa\opa(-\log |\delta|)$ on this $\Gamma_p$. 
Let us put $i\pa\opa(-\log |\delta|) = i\sum_{\alpha, \beta=1}^{n+1} L_{\alpha,\ol{\beta}} dz^{\alpha}\wedge d\ol{z}^\beta$ on $V \setminus M$.
Letting $\eta = 0$ in the computation in the proof of Theorem 1.1 of \cite{adachi_local} gives us 
\begin{align*}
L_{\alpha\ol{\beta}}(0,iy_{n+1}) &\to \frac{\d'^2 (-\log h)}{\d'z^\alpha \d'\ol{z}^\beta}(p), \\
y_{n+1} L_{\alpha,\ol{n+1}}(0,iy_{n+1}) &\to 0, \\
y^2_{n+1} L_{n+1,\ol{n+1}}(0,iy_{n+1}) &\to \frac{1}{4}
\end{align*}
as $(0, iy_{n+1}) \to (0,0) (= p)$ where $1 \leq \alpha, \beta \leq n$. 
Here we have used the fact that the converse of Brunella's construction in \cite{adachi_local} actually gives the inverse map: 
the hermitian metric of $N^{1,0}$ induced from $\delta$ agrees with the original $h^2$ thanks to the normalization of $\nu$. 

This limiting behavior shows that $i\Theta_h \oplus i\pa \delta \wedge \opa\delta/\delta^2$ is 
the main term of the asymptotics of $i\pa\opa(-\log |\delta|)$ along the normal line toward $p$. 
Comparing this asymptotics with the equality $\omega = (i\Theta_h \oplus i\pa \delta \wedge \opa\delta)/100$ on $M$
and using continuity of the forms, 
it is therefore possible to bound $i\pa\opa(-\log |\delta|)$ from below by $\omega$ on a sufficiently small $\Gamma_p$. 
\end{proof}

Now we have passed from given transversal measure $\mu$ of $\mathcal{F}$ 
to the defining function $\delta$ of $M$ with the aid of the holomorphic normal bundle $N^{1,0}$. 
Next, we are going to find a counterpart to the geometric information of the holonomy of $\mathcal{F}$
in the complement of $M$. 
Here we recall the finding of Ohsawa and Sibony in \cite{ohsawa-sibony1998}.

\begin{proposition}[cf. \cite{ohsawa-sibony1998}. See also \cite{harrington-shaw}]
\label{ohsawa-sibony}
There exists a $\Cont^3$-smooth real-valued function $\widetilde{\delta}$ defined on a neighborhood of $\ol{\Omega}$
such that 
\begin{enumerate}
\item Two functions $\delta$ and $\widetilde{\delta}$ agree on a neighborhood of $M$. 
\item There exists an $\eta \in (0,1]$ such that $i\pa\opa(-|\widetilde{\delta}|^\eta) \geq 0$ on $\Omega$ 
and $i\pa\opa(-|\widetilde{\delta}|^\eta) > 0$ on $W \cap \Omega$ where $W$ is a neighborhood of $M$. 
\end{enumerate}
\end{proposition}

\begin{proof}
By the argument in \cite{ohsawa-sibony1998} or \cite{harrington-shaw}, 
the existence of $\eta$ such that $i\pa\opa(-|\delta|^\eta) > 0$ on $W \setminus M$ 
follows from $i\pa\opa(-\log |\delta|) > \omega $ on $U \setminus M$ where $W \subset U$ is a smaller neighborhood of $M$.

We may suppose that $\delta > 0$ on $W \cap \Omega$ by changing the sign of $\delta$ if necessary. 
To extend $\delta$ to $\widetilde{\delta}$ on a neighborhood of $\ol{\Omega}$, we let
\[
\widetilde{\delta}(p) = \begin{cases}
\delta(p) & \text{for $p \in W$ with $\delta(p) \leq \varepsilon'/3$}, \\
(-\psi(-\delta(p)^\eta))^{1/\eta} & \text{for $p \in W$ with $\varepsilon'/3 \leq \delta(p) \leq \varepsilon'$}, \\
\varepsilon'/2 & \text{for other points of $\Omega$}. 
\end{cases}
\]
for sufficiently small $\varepsilon' > 0$ 
where $\psi: \R \to \R$ is a $\Cont^3$-smooth non-decreasing convex function such that 
$\psi(t) = t$ for $t \geq -(\varepsilon'/3)^\eta$ and $\psi(t) = -(\varepsilon'/2)^\eta$ for $t \leq -(\varepsilon')^\eta$. 
Then, one can see that $\widetilde{\delta}$ is the desired one. 
\end{proof}

The supremum of $\eta$ appearing in Proposition \ref{ohsawa-sibony} is denoted by $\eta_\delta$ 
and called the \emph{Diederich--Fornaess exponent} of $\delta$. 
The \emph{Diederich--Fornaess index} of a relatively compact domain $\Omega$ is defined to be the supremum of $\eta_\delta$ where 
we consider all the defining functions $\delta$ of $M = \bd\Omega$ satisfying the conditions in Proposition \ref{brunella}. 

This numerical index, the Diederich--Fornaess exponent $\eta_\delta$ of $\delta$, is the counterpart in $\Omega$ to
the ratio of $iA_\mu$ and $i\Theta_\mu$, a geometric information of the holonomy of $\mathcal{F}$.
The author showed in \cite{adachi_local} the following formula. 

\begin{proposition}[cf. \cite{adachi_local}]
\label{adachi}
We have an equality
\[
\eta_\delta = \eta_\mu \eqdef \min_{p \in M} \eta_\mu(p)
\]
where $\eta_\mu(p)$ is called the local Diederich--Fornaess exponent and defined by 
\[
\eta_\mu(p) \eqdef \sup \{ \eta \in (0,1) \mid i\Theta_\mu(p) - \frac{\eta}{1-\eta}iA_\mu(p) > 0 \}.
\]
\end{proposition}
\begin{proof}
By a direct computation exploiting the same technique used in the proof of Proposition \ref{brunella}. 
We refer the reader to \cite{adachi_local} for the detail.
\end{proof}

\section{Latter half of the proof: An estimate on the complement}

By Proposition \ref{adachi} and a simple computation, we now know that what we supposed is, in fact, equivalent to $\eta_\delta = \eta_\mu > 1/(n+1)$.
This is absurd from the following global estimate and we complete the proof.

\begin{proposition}[cf. \cite{adachi-brinkschulte} and \cite{fu-shaw}]
The Diederich--Fornaess exponent $\eta_\delta$ should not exceed $1/(n+1)$.
\end{proposition}

We recollect an elegant argument of Fu and Shaw \cite{fu-shaw} in our setting here. 
(See also Nemirovski{\u\i} \cite{nemirovski}. For another proof under a stronger assumption, see \cite{adachi-brinkschulte}.)

\begin{proof}
Suppose that $\eta_\delta > 1/(n+1)$. Then, we find $\eta > 1/(n+1)$ in Proposition \ref{ohsawa-sibony}. 
We assume $\delta > 0$ on $\Omega$ by taking its negative if necessary. 
Let $\omega_\eta = i\pa\opa(-\widetilde{\delta}^\eta)$, $\Omega_t = \{ p \in \Omega \mid \widetilde{\delta}(p) > t \}$ and 
\[
V(t) = \int_{\Omega_t} \omega_\eta^{n+1}.
\]
From the choice of $\eta$, $\omega_\eta^{n+1}$ defines a non-negative non-trivial measure on $\Omega$. 
Hence, it is clear that $V(t)$ is non-negative and non-increasing and, in particular, $V(0) \in (0, \infty]$. 

On the other hand, by Stokes' theorem, for $0 < t \ll 1$, 
\begin{align*}
V(t) 
&= \int_{\Omega_t} d( \frac{\pa - \opa}{2i}(-\widetilde{\delta}^{\eta})\wedge \omega_\eta^{n})\\
&= \int_{\bd\Omega_t} \frac{\pa - \opa}{2i}(-\delta^{\eta})\wedge \omega_\eta^{n}.
\end{align*}
By a direct computation, 
\begin{align*}
V(t) 
&= \int_{\bd\Omega_t} (\eta\delta^{\eta-1})^{n+1}  \frac{\pa - \opa}{2i}(-\delta) \wedge (i\pa\opa (-\delta))^n\\
&=  \eta^{n+1} \frac{t^{(n+1)\eta}}{t^{n+1}} \int_{\bd\Omega_t} \frac{\pa - \opa}{2i}(-\delta) \wedge (i\pa\opa (-\delta))^n. \\
&=  \eta^{n+1} \frac{t^{(n+1)\eta}}{t^{n+1}} \int_{\bd\Omega_t} \frac{\pa - \opa}{2i}(-\delta) \wedge (i\pa\opa (-\delta)|\Ker \pa\delta)^n. 
\end{align*}
Recall that the Levi-flatness is equivalent to the vanishing of the Levi form along $M$, 
$i\pa\opa (-\delta)|T^{1,0}\mathcal{F} = 0$ on $M$. We therefore have 
\[
V(t) = \eta^{n+1}  \frac{t^{(n+1)\eta}}{t^{n+1}} O(t^n) = O\left(\frac{t^{(n+1)\eta}}{t}\right)
\]
as $t \searrow 0$. Then, the choice of $\eta$ shows that $V(0) = 0$. This is a contradiction. 
\end{proof}

\section{An Example}
In this section, we illustrate our Theorem by an explicit example, a flat circle bundle over a compact Riemann surface.
We will consider a transverse measure $\mu$ having positive normal bundle curvature
and see that its local Diederich--Fornaess exponent distributes around $1/2$. 

Let $\Sigma$ be a compact Riemann surface of genus $\geq 2$. 
Fix an identification of its universal covering with the unit disk $\D$
and express $\Sigma = \D/\Gamma$ by a Fuchsian group $\Gamma$.
Let $\rho: \Gamma \to \Aut(\D)$ be a quasi-conformal deformation of $\Gamma$.
We suspend the unit circle $\bd\D$ over $\Sigma$ by $\rho$ and obtain a flat circle bundle, 
say $M_\rho = \Sigma \times_\rho \bd\D \to \Sigma$. Pulling back the complex structure of $\Sigma$ on each leaf,
we regard $M_\rho$ as an oriented compact real-analytic Levi-flat 3-manifold.
We can realize $M_\rho$ in a compact complex surface as a real-analytic Levi-flat real hypersurface. 
That is because we can suspend $\CP^1$ over $\Sigma$ by $\rho$, say $X_\rho = \Sigma \times_\rho \CP^1 \to \Sigma$ 
and $M_\rho$ is naturally identified with the boundary of the holomorphic disc bundle $\Omega_\rho = \Sigma \times_\rho \D \to \Sigma$.

Now we follow the idea of Diederich and Ohsawa \cite{diederich-ohsawa1985} and construct a transverse measure $\mu$ on $M_\rho$.
Let $s: \D \to \D$ be the $\rho$-equivariant harmonic diffeomorphism with respect to the Poincar\'e metric, that is, 
the $\rho$-equivariant solution to the Euler--Lagrange equation 
\[
s_{z\ol{z}} + 2\frac{\ol{s(z)}}{1-|s(z)|^2} s_z s_{\ol z} = 0.
\]
Note that $s$ defines a section of $\Omega_\rho \to \Sigma$. 
We consider a transverse measure $\mu$ given by 
\[
\mu = h(z, e^{i\theta})d\theta = \frac{1-|s(z)|^2}{|e^{i\theta} - s(z)|^2} d\theta.
\]
on a foliated chart given by the covering map $(z, e^{i\theta}) \in \D \times \bd \D \to M_\rho$. 
The equivariance of $s$ assures that $\mu$ descends to $M_\rho$. Note that the function $h$ is exactly the Poisson kernel. 

By a direct computation using the Euler--Lagrange equation and the harmonicity of the Poisson kernel, we have 
\begin{align*}
\frac{\partial}{\partial z} \log h
&= \frac{1}{1-|s(z)|^2}\left( \frac{1-e^{i\theta}\ol{s(z)}}{e^{i\theta}-s(z)} s_z(z) + \frac{1-e^{-i\theta}s(z)}{e^{-i\theta}-\ol{s(z)}} \ol{s_{\ol{z}}(z)} \right), \\
\frac{\partial^2}{\partial z \partial \ol{z}} (-\log h)
&= \frac{1}{(1 - |s(z)|^2)^2} \left| \frac{1-e^{i\theta}\ol{s(z)}}{e^{i\theta}-s(z)} s_z(z) - \frac{1-e^{-i\theta}s(z)}{e^{-i\theta}-\ol{s(z)}} \ol{s_{\ol{z}}(z)} \right|^2. 
\end{align*}
In a fiber over $z \in \Sigma$, if we choose its fiber coordinate so that $s(z) = 0$, we have 
\begin{align*}
iA_{\mu}(z, e^{i\theta}) &= \left| e^{-i\theta} s_z(z) + e^{i\theta} \ol{s_{\ol{z}}(z)} \right|^2 idz \wedge d\ol{z},\\
i\Theta_{\mu}(z, e^{i\theta}) &= \left| e^{-i\theta} s_z(z) - e^{i\theta} \ol{s_{\ol{z}}(z)} \right|^2 idz \wedge d\ol{z}
\end{align*}
and the expression of the local Diederich--Fornaess exponent $\eta_\mu$ of $\mu$ is given by
\[
\eta_\mu(z, e^{i\theta}) = \frac{i\Theta_\mu}{i\Theta_\mu + iA_\mu} = \frac{1}{2} - \frac{\Re (e^{-2i\theta}s_z s_{\ol{z}})}{|s_z|^2 + |s_{\ol{z}}|^2}.
\]

We remark that when $\rho = \id: \Gamma \hookrightarrow \Aut(\D)$, the harmonic diffeomorphism $s = \id: \D \to \D$ 
becomes biholomorphic, and $i\Theta_{\mu} = iA_{\mu}$ and 
$\eta_\mu = 1/2$ hold everywhere. It might be of interest that this $\mu$ corresponds to the foliated harmonic measure of $M_{\id}$.

\subsection*{\bf Acknowledgements.}
The author is grateful to a referee for his/her careful reading of the manuscript.


\begin{thebibliography}{99}
 \bibitem{adachi_local}
         M.\ Adachi, 
         \textit{A local expression of the Diederich--Fornaess exponent and the exponent of conformal harmonic measures}, 
         to appear in Bull. Braz. Math. Soc. (N.S.).
 \bibitem{adachi-brinkschulte}
         M.\ Adachi and J. Brinkschulte, 
         \textit{A global estimate for the Diederich--Fornaess index of weakly pseudoconvex domains},
         preprint.  

 \bibitem{barrett1990}
         D.\ E.\ Barrett,
         \textit{Complex analytic realization of Reeb's foliation of $S^3$},
         {Math. Z.}
         \textbf{203} (1990)
         355--361.
         
 \bibitem{barrett1992}
         D.\ E.\ Barrett,
         \textit{Global convexity properties of some families of three-dimensional compact Levi-flat hypersurfaces},
         {Trans. Amer. Math. Soc.}
         \textbf{332} (1992)
         459--474.
         
 \bibitem{barrett-fornaess1988}
         D.\ E.\ Barrett and J.\ E.\ Forn{\ae}ss, 
         \textit{On the smoothness of Levi-foliations},
         {Publ. Mat.}
         \textbf{32} (1988)
         171--177.

 \bibitem{barrett-inaba1992}
         D.\ E.\ Barrett and T.\ Inaba,
         \textit{On the topology of compact smooth three-dimensional Levi-flat hypersurfaces},
         {J. Geom. Anal.}
         \textbf{2} (1992)
         489--497.

 \bibitem{brunella2008}
         M.\ Brunella, 
         \textit{On the dynamics of codimension one holomorphic foliations with ample normal bundle},
         {Indiana Univ. Math. J.}
         \textbf{57} (2008)
         3101--3113.

 \bibitem{diederich-ohsawa1985}
         K.\ Diederich and T.\ Ohsawa, 
         \textit{Harmonic mappings and disc bundles over compact K\"ahler manifolds}, 
         {Publ. Res. Inst. Math. Sci.}
         \textbf{21} (1985),
         819--833.

 \bibitem{fu-shaw} 
         S.\ Fu and M.-C.\ Shaw,
         \textit{The Diederich-Forn{\ae}ss exponent and non-existence of Stein domains with Levi-flat boundaries}, 
         preprint. 

 \bibitem{harrington-shaw}
         P.\ S.\ Harrington and M.-C.\ Shaw,
         \emph{The strong Oka's lemma, bounded plurisubharmonic functions and the $\overline{\partial}$-Neumann problem}, 
         Asian J. Math. {\bf 11} (2007)
         127--139.

 \bibitem{inaba1992}
         T.\ Inaba,
         \textit{On the nonexistence of CR functions on Levi-flat CR manifolds},
         {Collect. Math.}
         \textbf{43} (1992)
         83--87.
         
 \bibitem{nemirovski}
         S.\ Y.\ Nemirovski{\u\i}, 
         \textit{Stein domains with {L}evi-plane boundaries on compact complex surfaces},
         {Mat. Zametki}
         \textbf{66} (1999)
         {632--635}.
		
 \bibitem{ohsawa-sibony1998}
         T.\ Ohsawa and N.\ Sibony, 
         \textit{Bounded p.s.h. functions and pseudoconvexity in K\"ahler manifold},
         {Nagoya Math. J.} 
         \textbf {149} (1998)
         1--8.

 \bibitem{rea1972}
         C.\ Rea, 
         \textit{Levi-flat submanifolds and holomorphic extension of foliations},
         Ann. Scuola Norm. Sup. Pisa 
         {\bf 26} (1972), 
         665--681. 
 \bibitem{ueda1982}
         T.\ Ueda, 
         \textit{On the neighborhood of a compact complex curve with topologically trivial normal bundle},
         J. Math. Kyoto Univ.
         \textbf{22} (1982/83),
         583--607.

\end{thebibliography}
\end{document}